\documentclass[11pt]{article}
\usepackage[utf8]{inputenc}
\usepackage[displaymath, tightpage]{preview}
\usepackage[all, cmtip]{xy}

\usepackage{fourier} 
\usepackage[scaled=0.875]{helvet} 

\usepackage{amsmath}
\usepackage{amsxtra}
\usepackage{amscd}
\usepackage{amsthm}
\usepackage{amsfonts}
\usepackage{amssymb}
\usepackage{bbm}
\usepackage{tikz}
\usepackage{tikz-cd}
\usetikzlibrary{matrix,calc,arrows}
\usepackage{hyperref}

\numberwithin{equation}{section}

\usepackage{xcolor}
    \usepackage{lipsum} 
\usepackage[draft]{todonotes}   

\newtheorem{theorem}{Theorem}
\newtheorem{lemma}[theorem]{Lemma}
\newtheorem{proposition}[theorem]{Proposition}
\newtheorem{definition}[theorem]{Definition}
\newtheorem{corollary}[theorem]{Corollary}

\numberwithin{theorem}{section}

\theoremstyle{definition}

\newtheorem{remark}[theorem]{Remark}

\newcommand\indlim\varinjlim

\def\cA{\mathcal{A}}

\def\cE{\mathcal{E}}

\def\cL{\mathcal{L}}

\def\cO{\mathcal{O}}

\def\cT{\mathcal{T}}

\newcommand\fm{{\mathfrak m}}

\newcommand\fS{{\mathfrak S}}

\newcommand\bbA{{\mathbb A}}

\newcommand\CC{{\mathbb C}}

\newcommand\PP{{\mathbb P}}
\newcommand\QQ{{\mathbb Q}}
\newcommand\RR{{\mathbb R}}

\def\bA{{\bbA}}

\def\bP{{\PP}}

\def\Z{\mathbf Z}
\def\Q{\mathbf Q}

\def\D{\mathbf D}
\def\N{\mathbf N}

\def\J{\mathrm J}

\def\M{\textsf M}

\def\c{{\mathrm c}}

\newcommand\GL{{\rm GL}}

\newcommand\tr{{\rm tr}}

\newcommand\pr{{\rm pr}}

\newcommand\id{{\rm id}}
\newcommand\Spec{{\rm Spec}}

\newcommand\Gm{\mathbb{G}_m}
\newcommand\Ga{{\mathbb{G}_a}}

\newcommand\val{{\rm val}}
\newcommand\Map{{\rm Map}}
\newcommand\Hom{{\rm Hom}}

\newcommand\End{{\rm End}}

\newcommand\Sym{{\rm Sym}}

\newcommand\Bun{{\rm Bun}}

\newcommand\Gr{{\rm Gr}}

\newcommand\ad{{\rm ad}}

\newcommand\Mat{{\rm Mat}}

\newcommand\Ql{{\QQ}_\ell}

\newcommand\Ind{{\rm Ind}}

\newcommand\cone{{\rm c}}
\newcommand\Prim{{\rm Prim}}

\newcommand\Spf{\mathrm{Spf}}
\newcommand\IC{\mathrm{IC}}

\newcommand{\punc}{*}

\usepackage{titling}

\title{On the formal arc space
of a reductive monoid}

\author{A.\ Bouthier, B.C.\ Ng\^o, Y.\ Sakellaridis}

\newcommand{\Addresses}{{ 
  \bigskip
  \footnotesize

  (A.\ Bouthier) \textsc{Einstein Institute of Mathematics, Hebrew University, Givat Ram, Jerusalem, 91904, Israel}\par\nopagebreak
  \textit{E-mail address}: \texttt{bouthier@math.huji.ac.il}

  \medskip

  (B.C.\ Ng\^o) \textsc{University of Chicago, 
Eckhart Hall, 5734 University Avenue,
Chicago, IL 60637,USA}\par\nopagebreak
  \textit{E-mail address}: \texttt{ngo@uchicago.edu}

  \medskip

  (Y.\ Sakellaridis) \textsc{Rutgers University - Newark, 101 Warren Street, Smith Hall 216, Newark, NJ 07102, USA, and Department of Mathematics, 
National Technical University of Athens,
Zografou 15780, Greece.}\par\nopagebreak
  \textit{E-mail address}: \texttt{sakellar@rutgers.edu}
}
}

\date{}

\begin{document}

\maketitle

\begin{abstract}
Let $X$ be a scheme of finite type over a finite field $k$, and let $\cL X$ denote its arc space; in particular, $\cL X(k) = X(k[[t]])$. Using the theory of Grinberg, Kazhdan, and Drinfeld on the finite-dimensionality of singularities of $\cL X$ in the neighborhood of non-degenerate arcs, we show that a canonical ``basic function'' can be defined on the non-degenerate locus of $\cL X(k)$, which corresponds to the trace of Frobenius on the stalks of the intersection complex of any finite-dimensional model. We then proceed to compute this function when $X$ is an affine toric variety or an ``$L$-monoid''. Our computation confirms the expectation that the basic function is a generating function for a local unramified $L$-function; in particular, in the case of an $L$-monoid we prove a conjecture formulated by the second-named author in \cite{Ngo-monoids}.
\end{abstract}

\begin{flushright}
\emph{Dedicated to the memory of Professor Igusa.}
\end{flushright}

\section*{Introduction}

The formal arc space $\cL X$ of an algebraic variety $X$ carries an important amount of information on singularities of $X$. Little is known about singularities of the formal arc scheme itself. According to Grinberg, Kazhdan \cite{GK} and Drinfeld \cite{Drinfeld} it is known, nevertheless, that the singularity of $\cL X$ at a non-degenerate arc is finite dimensional i.e.\ for every non-degenerate arc $x\in \cL X$, the formal completion of $\cL X$ at $x$ is isomorphic to $Y_y \times \D^\infty$ where $Y_y$ is the formal completion of a finite dimensional variety $Y$ at some point $y\in Y$ and $\D^\infty$ is the infinite power of the formal disc. 

One can hope to define the intersection complex of $\cL X$ via its local finite dimensional models and study the intersection complex as a measure of the singularity of $\cL X$. In this paper, we show that the trace of Frobenius function on the intersection complex is well defined on the space of non-degenerate arcs (to be defined in Section 1). The main result of this paper is the calculation of this function in the cases where $X$ is a toric variety or a special but important class of reductive monoids. 

The main motivation behind this calculation is an expectation that, at least when $X$ is an affine spherical variety under the action of a reductive group $G$, this function is, in a suitable sense, a generating series for an unramified local $L$-function (or product thereof). This expectation was stated in \cite{SaRS} in order to give a conceptual explanation to the Rankin-Selberg method, but the idea draws from the work of Braverman and Kazhdan who studied the Schwartz space of the basic affine space \cite{BKSchwartz}, and from relevant work in the geometric Langlands program \cite{BGEisenstein, BFGM}.

In the case when $X$ is in the class of reductive monoids that we term ``$L$-monoids'' (first introduced by Braverman and Kazhdan in \cite{BKgamma}), a precise conjecture was formulated in \cite{Ngo-monoids}. It states that this function, the trace of Frobenius on the intersection complex of the formal arc space, is the generating series of the local unramified $L$-function for the irreducible representation of the dual group whose highest weight determines the isomorphism class of the $L$-monoid. 
This generalizes the (local unramified) construction of Godement and Jacquet \cite{GJ} in the case $X=\Mat_n$. A proof of this conjecture is presented in the present paper.

In the case of affine toric varieties, the trace of Frobenius function on the intersection complex can be expressed as a generating series whose coefficients are the number of ways to decompose an element of the strictly convex cone of coweights defining the isomorphism class of the variety as a sum of its generators. In particular, these coefficients are natural numbers which are independent of the base field. This function can also be seen as the generating function for the product of local unramified $L$-functions of the torus determined by the generators in this coweight cone.

Since our method in the two cases are somewhat similar, one can hope to generalize the result to general spherical varieties.

\subsection*{Acknowledgments}
We thank D. Bourqui and J. Sebag for pointing out a mistake in the definition of finite dimensional formal model in an earlier version of our paper.

This work has been realized partly during the stay of the three authors at the MSRI in the fall 2014 and we would like to thank the Institute for its hospitality and the excellent working conditions.

The second named author has been supported by NSF grant 
DMS-1302819 and the Simons foundation. 
The third named author has been supported by NSF grant DMS-1101471.

\section{IC-function on the formal arc space} \label{section:IC-function}

Let $X$ be a scheme of finite type over a field $k$. For every positive integer $n$, we consider the $n$-arc functor $\J_n(X)$ whose $R$-points  are the $R[t]/t^{n+1}$-points of $X$. 
For $n=0$, $\J_0(X)=X$. For $n=1$, $\J_1(X)$ is the tangent bundle of $X$. 
If $X$ is affine, $\J_n(X)$ is representable by an affine scheme of finite type. It follows that in general, $\J_n(X)$ is representable as an affine $X$-scheme of finite type. For $m>n$, the truncation maps
$$p_n^m:\J_m(X) \to \J_n(X)$$
are thus affine. If $X$ is smooth, the maps $p_n^m$ are smooth and surjective for all $n>m$. 

The inverse limit $\cL X$ of $\J_nX$ is called the formal arc space. The set of its $R$-points is a projective limit
\begin{equation}
\cL X(R)= X(R[[t]]) =\varprojlim\limits_{n\to\infty} \J_n(R).
\end{equation}
For each integer $n$, we have a canonical map $p_n: \cL X \to \J_n(X)$. If $X$ is smooth, the maps $p_n$ are formally smooth and surjective. In the case where $X$ is not smooth, the geometry of $p_n$ is rather complicated. We refer \cite{EM} for a good account of this theory.

\begin{definition} 
A finite dimensional formal model of $\cL X$ at $x\in \cL X(k)$ is the formal completion $Y_y$ of a $k$-scheme of finite type $Y$ at a point $y\in Y$ equipped with an isomorphism of formal schemes
\begin{equation} \label{finite model}
(\cL X)_x \simeq Y_y \times \D^\infty
\end{equation}
where $\D$ is the formal disc.
\end{definition}

Let $X$ be an integral scheme over $k$. Let $X^\circ$ be a smooth dense open subset in $X$. Let us denote $Z$ the complement of $X^\circ$. 
We will denote $\cL^\circ X$ the space of arcs which generically map into $X^\circ$, i.e.\ for any test scheme $S$ we have $\cL^\circ X(S)=$ the set of maps $\phi: \D \times S \to X$ such that $\phi^{-1}(X^\circ)$ is an open $U\subset \D\times S$ surjecting to $S$. Such arcs will be called {\em non-degenerate} with respect to $X^\circ$. We have
\begin{equation} \label{circ}
\cL^\circ X(k)=\cL X(k) - \cL Z(k).
\end{equation}

According to Drinfeld \cite{Drinfeld} (and Grinberg, Kazhdan \cite{GK}), finite dimensional formal models exist for every point $x\in \cL^\circ X$.  

One should be able to use Drinfeld's theorem to define the notion of perverse sheaves over $\cL^\circ X$, and in particular the intersection complex of $\cL^\circ X$. In this paper, we will show a weaker statement: when $k$ is a finite field, one can define a canonical function on $\cL^\circ X(k)$ that has to be seen as the function of Frobenius trace on the sought-after intersection complex on $\cL X$. \footnote{Recently, after the completion of this paper, Bouthier and Kazhdan set up a foundation for a theory of perverse sheaves on the space of non-degenerate formal arcs, see \cite{BoK}.}

Because we are dealing with infinite dimensional schemes, we need to renormalize the cohomological shift in the construction of intersection complex. Let $X$ be a scheme of finite type over $k$. Let $U$ be a smooth open dense subscheme of $X$ with $U=\bigsqcup_i U_i$ where $U_i$ are the connected components of $U$. If $j_i:U_i \to X$ denotes the open embedding, then in the usual definition of \cite{BBD}, the intersection complex of $X$ would be the direct sum of $j_{i,!*}  \Ql[\dim(U_i)]$. In constrast with the usual definition, we set 
\begin{equation} \label{IC}
\IC_X=\bigoplus_i j_{i,!*}  \Ql.
\end{equation}
In our normalization, the restriction of $\IC_X$ to $U_i$ is the constant sheaf $\Ql$ placed on degree $0$ disregarding the dimension of $X$.

This naive normalization does not behave well with the Verdier duality, however it is more convenient in certain other aspects. If $p:Y\to X$ is a smooth morphism, then we have
\begin{equation}
\IC_Y= p^* \IC_X.
\end{equation}
Also, if $p:Y\to X$ is a finite morphism which is an isomorphism over a dense open subscheme of $X$, and if $Y$ is smooth, then we have
\begin{equation}
\IC_X=p_* {\Ql}_Y
\end{equation}
where ${\Ql}_{Y}$ is the constant sheaf of value $\Ql$ on $Y$. 

If $k$ is a finite field with $q$ elements, then the trace of the Frobenius operator on the stalk of the intersection complex of a scheme $X$ of finite type defines a function, to be denoted by the same symbol
\begin{equation}
\IC_X: X(k) \to \Ql.
\end{equation}
This function takes value $1$ on the $k$-points of smooth open subscheme of $X$. It can be regarded as a numerical invariant of singularity of $X$.

\begin{proposition} \label{IC-number}
Let $Y, Y'$ be $k$-schemes of finite type and $y\in Y(k), y'\in Y'(k)$ such that there exists an isomorphism of formal schemes
\begin{equation} \label{isom with infinitely many variables}
	Y_y \hat{\times} \D^\infty \simeq Y'_{y'}\hat{\times} \D^\infty.
\end{equation}
Then the equality 
\begin{equation}
	\IC_Y(y)=\IC_{Y'}(y')
\end{equation}
holds.
\end{proposition}

\begin{proof}
It will be enough to prove that there exists a formally smooth morphism  
\begin{equation} \label{formally smooth with finitely many variables}
Y'_{y'} \times \D^{m'} \to Y_y.
\end{equation}  
If we assume \eqref{formally smooth with finitely many variables} exists, then according to \cite[Prop. 2.5(i)]{Schlessinger} there exists an integer $m$ such that \eqref{formally smooth with finitely many variables} factors through an isomorphism 
\begin{equation}
Y'_{y'} \times \D^{m'} \simeq Y_y \times \D^m.
\label{artin}
\end{equation}
It follows from Artin's approximation \cite[Cor.2.6]{Artin} that if two pointed $k$-schemes $(X,x)$ and $(Y,y)$ of finite type have isomorphic formal neighborhoods, then there exist \'etale neighborhoods $(U,x)$ of $x$ in $X$ and $(V,y)$ of $y$ in $Y$ such that $(U,x)$ and $(V,y)$ are isomorphic. (One should be aware that this isomorphism is not canonical, and in particular the isomorphism it induces between formal completions may not coincide with the given one, but can be congruent to it to an arbitratry high order.) In the situation \eqref{artin}, we infer that $Y'\times \bA^{m'}$ and $Y\times \bA^m$ have isomorphic \'etale neighborhoods at $(y',0)$ and $(y,0)$ respectively. The stalks of the intersection complexes of $Y$ and $Y'$ at $y$ and $y'$ are therefore isomorphic, and in particular, they have the same trace under the Frobenius operator. 

We now prove the existence of \eqref{formally smooth with finitely many variables}. 
 
The map is in fact easy to describe: it will be the bottom map in the diagram \begin{equation} \label{composition}
\begin{tikzcd}
Y_{y'} \hat\times \D^\infty \arrow{r}{\phi_+} 
& Y_y \hat\times \D^\infty \arrow{d}{} \\
Y_{y'} \times \D^{m'} \arrow{u}[swap]{} \arrow{r}[swap]{\phi'}
& Y_y
\end{tikzcd}
\end{equation}
where the left vertical map is induced from an embedding of a finite-dimensional formal disk into the infinite formal disk, the upper horizontal map is \eqref{isom with infinitely many variables}, and the right vertical map is the canonical projection. What we need to prove is that for $m'$ large enough, $\phi'$ is formally smooth.

 We can assume that $Y$ is a closed subscheme of the affine space $\bA^n$ with coordinates $x_1,\ldots,x_n$ defined by the equations $f_1,\ldots,f_r \in k[x_1,\ldots,x_n]$. In other words, we have $Y=\Spec(A)$ with
\begin{equation*}
	A=k[x_1,\ldots,x_n]/(f_1,\ldots,f_r).
\end{equation*}
We will also also that $y$ correspond to the maximal ideal $\fm_A$ of $A$ generated by $x_1,\ldots,x_n$. We denote $\overline A$ the completion of $A$ with respect to $\fm_A$.

We consider the polynomial $A$-algebra $A_+=A[u_1,u_2,\ldots]$ where the variables $u_1,u_2,\ldots$ are countably infinite in number. Let $\fm_{A_+}$ be the maximal ideal of $A_+$ generated by $x_1,\ldots,x_n$ and $u_1,u_2,\ldots$, and let $\overline{A_+}$ denote the completion of $A_+$ with respect to $\fm_{A_+}$:
\begin{equation*}
	\overline{A_+}=\projlim_n A_+/\fm_{A_+}^{h}.
\end{equation*} 
We have
\begin{equation*}
	Y_y\hat\times\D^\infty =\Spf(\overline {A_+}).
\end{equation*}

If we denote $\overline {\fm_{A_+}^{h}}$ the kernel of the homomorphism 
$\overline{A_+} \to A_+/\fm_{A_+}^{h}$, then we have an isomorphism 
\begin{equation}
	A_+/\fm_{A_+}^{h} \to \overline{A_+}/\overline{\fm_{A_+}^{h}}.
\end{equation}
This induces for every $h$, an isomorphism 
\begin{equation} \label{graded}
	\fm_{A_+}^{h}/\fm_{A_+}^{h+1} \to \overline{\fm_{A_+}^{h}}/\overline{\fm_{A_+}^{h+1}}.
\end{equation}
In particular, for $h=1$, $\overline{\fm_{A_+}^{1}}/\overline{\fm_{A_+}^{2}}$ is an infinite dimensional vector space generated by the images of $x_1,\ldots,x_n$ and $u_1,u_2,\ldots$, in other words every element of $\overline{\fm_{A_+}^{1}}$ is congruent modulo $\overline{\fm_{A_+}^{2}}$ to a finite linear combination of  $x_1,\ldots,x_n$ and $u_1,u_2,\ldots$.

Similarly, $Y'=\Spec(B)$ with 
\begin{equation*}
	B=k[y_1,\ldots, y_{n'}]/(g_1,\ldots,g_{r'}).
\end{equation*}
and $y\in Y(k)$ corresponds to the maximal ideal $\fm_B$ generated by $y_1,\ldots,y_{n'}$. We denote $\overline B$ the completion of $B$ with respect to $\fm_B$.

We also denote by $B_+$ the polynomial algebra $B[v_1,v_2,\ldots]$ in countably many variables $v_1,v_2,\ldots$, $\fm_{B_+}$ the maximal ideal generated by $y_1,\ldots,y_{n'}$ and $v_1,v_2,\ldots$, and $\overline{B_+}$ the completion of $B_+$ with respect to $\fm_{B_+}$. The complete ring $\overline{B_+}$ is also filtered by the ideals $\overline {\fm_{B_+}^h}$ with
\begin{equation*}
	\overline {\fm_{B_+}^h}/\overline {\fm_{B_+}^{h+1}} \simeq
	{\fm_{B_+}^h}/{\fm_{B_+}^{h+1}}.
\end{equation*}
We have
\begin{equation*}
	Y'_{y'}\hat\times \D^\infty = \mathrm{Spf}(\overline{B_+}).
\end{equation*}
 
The isomorphism \eqref{isom with infinitely many variables} induces an isomorphism $\phi_+:\overline{A_+} \to \overline{B_+}$. Let us denote $\phi_i=\phi(x_i)$ for all $i=1,\ldots,n$. There exists an integer $m'$ such that the images of $\phi_i$ in $\fm_{B_+}/\fm_{B_+}^2$ are linear combinations of images of $y_1,\ldots,y_{n'}$ and $v_1,v_2,\ldots,v_{m'}$. For this integer $m'$, we claim that the induced morphism $\phi'$ in diagram \eqref{composition} is formally smooth.

In algebra, the morphism $Y'_{y'}\times \D^{m'} \to Y_y$ corresponds to the homomorphism
\begin{equation}
	\phi':\overline A\to \overline{B}[[v_1,v_2,\ldots,v_{m'}]]
\end{equation}
given by $x_i \mapsto \phi'(x_i)$ where $\phi'(x_i)$ is the formal series obtained from $\phi_+(x_i)$ by setting $v_j=0$ for all $j>m'$. 
In order to prove that $\phi'$ is formally smooth, it is enough to construct another isomorphism 
\begin{equation} \label{new isom with infinitely many variables}
	\phi'_+:Y_{y'} \hat\times \D^\infty \to Y_y \hat\times \D^\infty,
\end{equation}
such that the diagram 
\begin{equation}
\begin{tikzcd}
Y_{y'} \hat\times \D^\infty \arrow{r}{\phi'_+} \arrow{d}[swap]{}
& Y_y \hat\times \D^\infty \arrow{d}{} \\
Y_{y'} \times \D^{m'} \arrow{r}[swap]{\phi'}
& Y_y
\end{tikzcd}
\end{equation}
in which vertical maps are canonical projections, is commutative. Indeed, the formal smoothness of $\phi'$ would then follow from the formal smoothness of the three other maps in this diagram.

In algebra, $\phi'_+:\overline{A_+} \to \overline{B_+}$ is given by ${\phi'}_+(x_i)=\phi'(x_i)$ and ${\phi'}_+(u_j)=\phi_+(u_j)$. It remains to prove that ${\phi'}_+:\overline{A_+} \to \overline{B_+}$ is an isomorphism of complete algebras.
By construction we have
\begin{equation} \label{congruence}
	\phi_+(x_i)\equiv \phi'_+(x_i) \mod \overline{\fm_{B_+}^2}.
\end{equation}
We observe that $\phi'_+(\fm_{A_+})\subset \overline{\fm_{B_+}}$. It follows $$\phi'_+(\overline{\fm_{A_+}^h})\subset \overline{\fm_{B_+}^h}$$ for all $h$. It is now enough to prove that the induced morphism 
\begin{equation}
\mathrm{gr}_h({\phi'}^+):	\overline{\fm_{A_+}^h} / \overline{\fm_{A_+}^{h+1}} \to 
\overline{\fm_{B_+}^h} / \overline{\fm_{B_+}^{h+1}}
\end{equation}
is an isomorphism for all $h\in\N$. 

Because of \eqref{graded} and \eqref{congruence}, $\phi^+$ and ${\phi'}^+$ induce the same map on graded pieces; in other words, the equality 
$$\mathrm{gr}_h({\phi'}^+)=\mathrm{gr}_h({\phi}^+)$$ 
holds for every $h\in\N$. Now since $\phi_+:A_+\to B_+$ is an isomorphism, $\mathrm{gr}_h({\phi'}^+)$ and $\mathrm{gr}_h({\phi}^+)$ are isomorphisms between graded pieces, and thus  ${\phi'}^+:A_+ \to B_+$ is also an isomorphism.
\end{proof}

It follows from this proposition that we have a well defined function on the set of of non-degenerate arcs
\begin{equation}
\IC_{\cL X}: \cL^\circ X(k) \to \Ql.
\end{equation}

\section{Global model for the formal arc space of a group embedding}
In the case of group embeddings, one can construct a finite-dimensional formal model at points of the formal arc space by constructing a moduli problem for bundles over a smooth projective curve with additional data. Let $X$ be an affine normal integral variety over a field $k$ equipped with an open embedding of a reductive group $G\hookrightarrow X$ and an action of $G\times G$ which extends the action on $G$ by left and right multiplication. This action automatically extends to $X$, making it into a monoid. For the purposes of defining the space $\cL^\circ X$ according to the previous section, we take $X^\circ$ to be the image of $G$ in $X$.

We will consider the algebraic stack $[G\backslash X/G]$ whose value on each test scheme $S$ is the groupoid of pairs of (left) principal $G$-bundles $E$, $E'$ on $S$ equipped with a $G$-equivariant morphism: $\phi:S\to X\wedge^{G\times G} (E\times E')$, where by $\wedge^{G\times G}$ we denote the quotient of the product by the diagonal action of $G\times G$.
 
Such a section $\phi$ will be called an $X$-morphism from $E$ to $E'$. Since $X$ is equipped with the structure of a monoid, $X$-morphisms between $G$-bundles can be composed.

Let $C$ be a smooth projective geometrically connected curve over $k$. We fix a principal $G$-bundle $E_0$ of $G$ over $C$, which will serve as our $G$-bundle of reference. We consider the the stack $\Map(C,[G\backslash X/G])$ of all maps from $C$ to the quotient stack $[G\backslash X/G]$; according to \cite{LM}, $\Map(C,[G\backslash X/G])$ is an algebraic stack locally of finite type. Over each test scheme $S$, an object of $\Map(C,[G\backslash X/G])$ is a map $\phi:C\times S \to [G\backslash X/G]$, in other words a $X$-morphism $\phi:E\to E'$ between two principal $G$-bundles $E$ and $E'$ over $C\times S$. Such map is said to be {\em non-degenerate} if $\phi^{-1}([G\backslash G/G])$ is an open subset of $C\times S$ whose projection on $S$ is surjective.
We will denote $\Map^\circ(C,[G\backslash X/G])$ the open substack of $\Map(C,[G\backslash X/G])$ consisting of non-degenerate maps only.

The stack $\Map(C,[G\backslash X/G])$ comes equipped with two maps to the moduli stack $\Bun_G$ of principal $G$-bundles on $C$: the ``left'' and the ``right'' one. We denote by $\Map(C,[G\backslash X/G])_0$ the fiber of the left one over $E_0$ (i.e.\ the base change with respect to the map $\operatorname{pt}\to \Bun_G$ defined by $E_0$). When $E_0$ is trivial, this is just the stack $\Map(C,[X/G])$. Except when the contrary is expressly mentioned, {\em $E_0$ will be the trivial $G$-bundle}.
We will consider the open substack 
\begin{equation}
M=\Map^\circ(C,[G\backslash X/G])_0
\end{equation}
of non-degenerate $X$-morphisms $E_0\to E$ between principal $G$-bundles over $X$. In later sections, we will show that, in some cases of interest, $M$ is an algebraic space locally of finite type.

We assume that $C(k)\ne \emptyset$, and fix once and for all a $k$-point $v\in C(k)$, an identification of its formal neighborhood $C_v$ with the formal disk $\D$. 
 
We will denote $C-\{v\}$ by $C'$.
We consider the stack $\tilde M$ classifying pairs $(\phi, \xi)$ where $\phi$ is a point of $M$ corresponding to an $G$-torsor $E$ on $C$ and $\xi$ is a trivialization of the restriction of $E$ to the formal completion $C_v$. Points of $\tilde M$ over a test scheme $S$ consist in a principal $G$-bundle $E$ over $C\times S$, a morphism $\phi:E_0\times S \to X\wedge^G E$ which induces an isomorphism between $E_0\times U$ and $E|_U$ over an open $U\times C\times S$ surjecting to $S$, and a trivialization of the underlying $G$-bundle $E$ over $C_S$ on the formal completion $(C_S)_v$ of $C_S$ along $\{v\}\times S$. We have the canonical projection
\begin{equation} \label{tilde-M-M}
\pi:\tilde M \to M
\end{equation}
which is a torsor under the group $\cL G$. 

Restricted to $(C_S)_v\simeq \D\hat\times S$, and taking into account the fixed trivialization of $E_0$ over $\D$, $(\phi,\xi)$ induces a morphism: $(C_S)_v \to X$ such that the preimage of $X^\circ=G$ is an open subset $U \subset (C_S)_v$ whose projection on $S$ is surjective. Thus we have a morphism 
\begin{equation} \label{global-to-local}
h:\tilde M \to \cL^\circ X.
\end{equation}

This morphism is not formally smooth, because of singularities introduced when a map $\phi$ as above has image in the singular locus of $[X/G]$ at a point $v'\ne v$. For our purposes, though, we only need to look at the formal neighborhood of a point $(\phi,\xi)\in \tilde M$ such that the only singularity of $\phi$ is at $v$.

\begin{proposition} \label{formallysmooth}
Let $\tilde m=(\phi,\xi)\in \tilde M(k)$ have the property that over $C'$ the map $\phi$ has image in $X^\bullet$, where $X^\bullet$ denotes the smooth locus of $X$. Then the morphism of formal neighbourhoods induced by \eqref{global-to-local}:
$$\tilde M_{\tilde m} \to \cL^\circ X_{x},$$
where $x\in \cL^\circ X(k)$ is the image of $\tilde m$ under $h$, 
is formally smooth.
\end{proposition}

\begin{proof}
The formal smoothness of formal neighbourhoods can be proved by the lifting property of points with values in local artinian rings. Let $(R,\mathfrak m)$ be a local artinian $k$-algebra and $\bar R=R/I$ where $I$ is an ideal of $R$ with $I^2=0$. We will denote $S=\Spec R$ and $\bar S=\Spec \bar R$. Let $(\bar\phi,\bar\xi)$ denote an $\bar R$-point of $\tilde M$ whose reduction modulo $\mathfrak m/I$ is equal to $\tilde m$. (Hence, necessarily, $R/\mathfrak m=k$.) Denote its 
image in $\cL^\circ X$ by $\bar \phi_v$ -- its reduction modulo $\mathfrak m$ is equal to $x$. Let $\phi_v$ be an $R$-point of $\cL^\circ X$ lifting $\bar\phi_v$. The morphism \eqref{global-to-local} is formally smooth if and only if in each situation as above, there exists an $R$-point $(\phi,\xi)$ of $M$, which lifts $(\bar \phi,\bar \xi)$ and maps to $\phi_v$. 

According to a variant of the Beauville-Laszlo formal patching theorem due to Heinloth \cite{Heinloth}, the $\bar R$-point 
$(\bar\phi,\bar\xi)$ of $\tilde M$ corresponds to the following collection of data:
\begin{itemize}
\item a formal arc $\bar\phi_v:(C_{\bar S})_v \to X$ which is non-degenerate with respect to $G=X^\circ$;
\item a morphism $\bar\phi': C'_{\bar S} \to [X/G]$ inducing a principal $G$-bundle $\bar E'$ on $C' \times \bar S$; notice that $\bar\phi'$ will necessarily factor through $[X^\bullet/G]$, since this is the case for its reduction modulo $\mathfrak m/I$;
\item a trivialization $\bar \beta$ of $\bar E'$ over the "punctured formal disc" $(C_{\bar S})_v^\punc$ such that the equality $\bar \beta^*(\bar \phi')=\bar\phi_v$ holds over $(C_{\bar S})_v^\punc$.
\end{itemize}
Here the ``punctured formal disc'' $(C_{\bar S})_v^\punc$ is defined to be the cartesian product of $(C_{\bar S})_v$ and $C'_{\bar S}$ over $C_{\bar S}$. The meaning of the equality $\bar \beta^*(\bar \phi')=\bar\phi_v$ also needs to be unraveled: $\phi'$ is given as a section of $[X^/G]$ giving rise to a principal $G$-bundle $E'$, which is trivialized over $(C_{\bar S})_v^\punc$, thus $\phi'$ defines via this trivialization a section $\bar \beta^*(\bar \phi')$ of $X^\bullet$ over $(C_{\bar S})_v^\punc$, to be compared with the restriction of $\bar\phi_v$.

Similarly, the $R$-point $(\phi,\xi)$ of $\tilde M$ is equivalent to the following collection of data
\begin{itemize}
\item a formal arc $\phi_v:(C_S)_v \to X$ which is non-degenerate with respect to $X^\circ$;
\item a morphism $\phi': C'_S \to [X/G]$ inducing a principal $G$-bundle $E'$ on $C' \times S$; as above, $\phi'$ will actually factor through $[X^\bullet/G]$;
\item a trivialization $\beta$ of $E'$ over the "punctured formal disc" $(C_S)_v^\punc$, such that the equality $\beta(\phi')=\phi_v$ holds over $(C_S)_v^\punc$. 
\end{itemize}

Among the above list of data, the first item $\phi_v$ is given. We will need to construct $\phi'$ and $\beta$ lifting $\bar\phi'$ and $\bar\beta$ respectively. First, as $X^\bullet$ is smooth, the map $\bar\phi': C'_{\bar S} \to [X^\bullet/G]$ can be lifted to a map $\phi': C'_S \to [X^\bullet/G]$. This map gives rise to a principal $G$-bundle $E'$ on $C'_S$. 

Consider the restriction $E'_v$ of $E'$ to the punctured formal disc $(C_S)_v^\punc$. Its reduction to $\bar R$, denoted by $\bar E'_v$, is trivialized by $\bar\beta$. By smoothness of $G$-bundle, there exists a trivialization $\beta'$ of the $G$-bundle $E'_v$ over $(C_S)_v^\punc$, extending $\bar\beta$. 

A priori, the restrictions of $(\beta')^*(\phi')$ and $\phi_v$ to $(C_S)_v^\punc$ define two different sections of $X$ which however  coincide on $(C_{\bar S})_v^\punc$. We will need to correct the trivialization so that the equality $\beta(\phi')=\phi_v$ occurs over $(C_S)_v^\punc$. This amounts to the existence of $\alpha\in G(R((t)))$ such that 
\begin{equation} \label{alpha-corrector}
\beta(\phi')=\alpha(\phi_v)
\end{equation}
which moreover maps to identity in $G(\bar R((t)))$.

It is enough to prove that the restrictions of $(\beta')^*(\phi')$ and $\phi_v$ to $(C_S)_v^\punc$ define sections $(C_S)_v^\punc \to X$, which factor through $G$. Indeed, in that case the element $\alpha \in G(R((t)))$ satisfying \eqref{alpha-corrector} would exist uniquely. This would prove the existence of $\alpha \in G(R((t)))$ satisfying \eqref{alpha-corrector} and the equality $\bar\alpha=1$ in $G(\bar R((t)))$ simultaneously. 

In order to prove that these maps factor through $G$, we use the crucial assumption that $R$ is an artinian local ring. Under this assumption, the underlying topological space of $(C_S)_v^\punc$ has just one point since its reduced scheme is $\Spec k((t))$.  
As the image of these maps in $X$ is not contained in $X-G$, being just a point, it is entirely contained in $G$. \end{proof}

Suppose $k$ is a finite field. Let $m\in M(k)$ and $x\in \cL^\circ X(k)$ such that there exists $\tilde m\in \tilde M(k)$ satisfying the assumptions of Proposition \ref{formallysmooth} and such that $\pi(\tilde m)=m$ and $h(\tilde m)=x$ where $\pi:\tilde M\to M$ is the map \eqref{tilde-M-M} and $h:\tilde M\to \cL X$ is the map \eqref{global-to-local}. Since $\tilde M_{\tilde m} \to (\cL X)_x$ is formally smooth, and $M$ is an algebraic space locally of finite type in the cases of interest, $M_m$ is a finite dimensional formal model of $(\cL X)_x$. According to Proposition \ref{IC-number}, there is an equality of numbers 
\begin{equation} \label{global-local}
\IC_M(m)=\IC_{\cL X}(x).
\end{equation}  
Finally, we prove that we can always find $m$ and $\tilde m$ as in the previous proposition.

\begin{proposition} \label{existence-of-global-point}
For any $x\in \cL^\circ X(k)$, there exists $\tilde m\in \tilde M(k)$ satisfying the assumptions of Proposition \ref{formallysmooth},  with $h(\tilde m) = x$. 
\end{proposition}

\begin{proof}
Since $\cL^\circ X(k)= X(\cO) \cap G(F)$, an element $x\in \cL^\circ X(k)$ defines a $k$-point in the affine Grassmannian $G(F)/G(\cO)$. In other words, $x$ gives rise to a $G$-bundle $\cE$ on the formal disc $C_v=\D$ equipped with a trivialization over the punctured formal disc $\D^\punc$. By glueing with the trivial $G$-bundle on $C'=C-\{v\}$, we obtain, according to Beauville-Laszlo \cite{BL}, a $G$-bundle $E$ on $C$ with a section $\phi'$ over $C'$. The assumption $x\in X(\cO) \cap G(F)$ implies that $\phi'$ can be extended as a section $\phi:C \to M\wedge^G E$. Thus we have constructed a point $m=(E,\phi) \in M(k)$ satisfying the assumptions of Proposition \ref{formallysmooth}. By construction, there exists a unique point $\tilde m\in \tilde M(k)$ over $m$ such that $h(\tilde m)=x$.
\end{proof}

\section{Toric case}

In this section, $X$ will be a affine normal toric variety, and $G$ will be a split $k$-torus $T$. Let $\Lambda_*(T)=\Hom(\Gm,T)$ denote the group of cocharacters of $T$ and $\Lambda^*(T)$ the group of characters. We recall that a toric variety is a pair $(X,T)$ where $X$ is a algebraic variety containing a torus $T$ as an open dense subset such that the action of 
$T$ on itself by translation can be extended as an action of $T$ on $X$. In case where there is no confusion to be feared, we use the letter $X$ to denote the toric variety $(X,T)$. We will only consider toric varieties $X$ which are affine, and whose torus is split. 

The affine toric variety $X$ is determined by the strictly convex (i.e.\ not containing lines) cone in $\Lambda_*(T)\otimes \RR$ generated by the monoid $\cone$ of cocharacters $\lambda\in \Lambda_*(T)$ such that $\lim_{t\to 0} \lambda(t)$ exists in $X$. This monoid is finitely generated and, by normality, \emph{saturated} inside of $\Lambda_*(T)$; in this section, we will be using the term ``monoid'' for finitely generated submonoids of torsion-free abelian groups, and we will be saying that a monoid $\cone$ is ``saturated'' inside of an abelian group $\Lambda$ if $n\lambda\in \cone$ for $n\ge 0$ implies that $n\in\cone$.

We can also reconstruct $X$ from its ring of regular functions
\begin{equation}
k[X]=\Spec k\left[e^\alpha \mid \alpha \in \cone^*\right],
\end{equation}
where 
\begin{equation}
\cone^*=\left\{ \alpha\in \Lambda^*(T) | \langle \lambda,\alpha\rangle \geq 0, \forall \lambda\in \cone \right\}
\end{equation}
and $e^\alpha$ denotes the regular function $T\to \bA^1$ attached to the character $\alpha:T\to \Gm$.

Let $\cT$ denote the category whose objects are normal, affine toric varieties, and morphisms are morphisms of tori which extend to the toric varieties. The above construction 
\begin{equation}
(X,T) \mapsto (\cone,\Lambda_*(T))
\end{equation}
that associates a toric variety $(X,T)$ to the saturated submonoid $\cone$ of $\Lambda_*(T)$, defines an equivalence of categories from $\cT$ to the category of pairs $(\cone,\Lambda)$ where $\Lambda$ is a  finitely generated free abelian group and $\cone \subset \Lambda$ is its intersection with a strictly convex cone generated by finitely many elements of $\Lambda$. 

In particular, for every normal affine toric variety $X$, each element $\lambda\in\cone$ corresponds to a morphism $\Gm\to T$ which extends to a morphism $\Ga\to X$. In other words, each element $\lambda\in\cone$ corresponds to a morphism $(\Ga,\Gm) \to (X,T)$. In the opposite direction, an element $\alpha\in \cone^*$ corresponds to a morphism $(X,T) \to (\Ga,\Gm)$.   

The formal arc space of $X$ is the functor $\cL X$ that associates to every $k$-algebra $R$ the set $X(\cO_R)$ where $\cO_R=R[[t]]$ is the ring of formal series of variable $t$ with coefficients in $R$. As is section \ref{section:IC-function}, we will only consider the subfunctor $\cL^\circ X$ of $\cL X$ consisting of non-degenerate formal arcs with respect to the smooth open subset $X^\circ=T$. For every $k$-algebra $R$, $\cL^\circ X(R)$ is the set of $x:\Spec (\cO_R) \to X$ such that the projection from the open subset $U=x^{-1}(T)$ to $\Spec(R)$ is surjective. This subfunctor is represented by an open subscheme of $\cL X$ which is the complement of $\cL(X-T)$. 

We observe that every morphism $(X,T) \to (X',T')$ in the category $\cT$ induces a morphism $\cL X \to \cL X'$ and $\cL^\circ X \to \cL^\circ X'$. In particular, for $\lambda\in \cone$, we have a corresponding morphism of toric varieties $\lambda:(\Ga,\Gm)\to (X,T)$. We will denote 
\begin{equation}
t^\lambda= \lambda(t)
\end{equation}
the image of $t\in \cO\cap F^\times$ in $X(\cO)\cap T(F)$, where $\cO=k[[t]]$ and $F=k((t))$. We have a canonical bijection
\begin{equation}
\cL^\circ X(k)= X(\cO)\cap T(F).
\end{equation}
The orbits of $T(\cO)$ on $\cL^\circ X(k)$  are parametrized by the monoid $\cone$:
\begin{equation}\label{orbit-torus}
X(\cO)\cap T(F) = \bigsqcup_{\lambda\in \cone} T(\cO)t^{\lambda}.
\end{equation}
Indeed, $X(\cO)\cap T(F)$ is a $T(\cO)$-invariant subset of 
$T(F)=\bigsqcup_{\lambda\in \Lambda} T(\cO)t^{\lambda}$, and by very definition of $\cone$, we have $t^\lambda\in X(\cO)$ if and only if 
$\lambda\in\cone$.

\begin{theorem} \label{local-equality-toric}
Suppose that $k$ is a finite field.
The $\IC$ function of $\cL^\circ X$ is $T(\cO)$-invariant and can be identified with a formal series with exponents in $\cone$:
\begin{equation}
\IC_{\cL^\circ X}=\sum_{\lambda\in \cone} m_\lambda e^\lambda
 \in \Ql[[\cone]]
\end{equation}
where $m_\lambda$ denotes the value of the $\IC$-function at $t^\lambda$.

Let $\Prim(\cone)$ be the set of primitive elements of $\cone$ i.e nonzero elements $\mu\in\cone$ that cannot be  decomposed as a sum $\mu=\mu_1+\mu_2$ where $\lambda_1,\lambda_2$ are nonzero elements of $\cone$. Then the equality
\begin{equation}
\IC_{\cL^\circ X} = \prod_{\nu\in\Prim(\cone)} (1-e^\nu)^{-1}.
\end{equation}
holds in $\Ql[[\cone]]$. In other words, for every $\lambda\in\cone$, $m_\lambda$ is the number of ways to decompose $\lambda$ as sum of primitive elements. 
\end{theorem}

To prove the theorem, we consider the global  analogue of $\cL X$. Let $C$ be a smooth projective geometrically connected curve defined over $k$. We consider the functor $M=M_X$ on the category of $k$-algebras which associates to every $k$-algebra $R$ the groupoid of maps $\phi: C_R\to [X/T]$ such that the preimage $\phi^{-1}([T/T])$ of $[T/T]$ is an open subset $U \subset C_R$ whose projection on $\Spec(R)$ is surjective. The association $X\to M=M_X$ is functorial in $\cT$.

\begin{proposition} \label{representable}
The functor $R\mapsto M(R)$ defined as above is representable by countable disjoint union of projective schemes over $k$. 
\end{proposition}

\begin{proof}
First, we consider the case where $T=\Gm$ and $X=\bA^1$. In this case, $M$ classifies pairs $(E,\phi)$ where $E$ is a invertible sheaf over $C$ and $\phi$ is a nonzero global section. The zero divisor of $\phi$ defines a point of the $n$-th symmetric power $C_n$ of $C$ where $n$ is the degree of $E$. It is well known that this induces an isomorphism of functors
\begin{equation}
M = C_\N = \bigsqcup_{n\in\N} C_n
\end{equation}
between $M$ and the disjoint union of symmetric powers of $C$.

Next, we consider the case where $\cone=\N^r$ is a free monoid. In this case $T=\Gm^r$ and $X=\bA^r$. We derive from the $\bA^1$-case that
\begin{equation} \label{M-free-monoid}
M= (C_\N)^r
\end{equation}

Now we consider the general case of a toric variety $(X,T)$ corresponding to a pair $(\cone,\Lambda)$ formed by a saturated strictly convex monoid $\cone$ inside a finitely generated free abelian group $\Lambda$. There is a canonical way to embed $X$ the moduli space $M$ for $(X,T)$ into the moduli space $M$ for a free monoid.

We consider the dual monoid $\cone^* \subset \Lambda^*$ consisting of elements $\alpha$ in the dual abelian group $\Lambda^*$ which take nonnegative values on $\cone$. Let $P \subset \cone^*$ be a set of elements that generate $\cone^*$ \emph{as a monoid}. We denote $\cone_P^*=\N^P$ and $\Lambda_P^*=\Z^P$ respectively the free monoid and the free abelian group generated by $P$. There is a canonical surjective map of monoids $\cone_P^* \to \cone^*$ and of abelian groups $\Lambda_P^* \to \Lambda^*$. By duality we have a monoid $\cone_P \subset \Lambda_P$ with injective maps $\cone \to \cone_P$ and $\Lambda\to \Lambda_P$, and closed embeddings $T\hookrightarrow T_P$, $X\hookrightarrow X_P$.

The pair $(\cone_P,\Lambda_P)$ corresponds to a toric variety $(X_P,T_P)$ with $T_P\simeq \Gm^r$ and $X_P\simeq \bA^r$ where $r$ is the cardinality of $P$. Let us denote $M_P$ the moduli space $M$ corresponding to the pair $(X_P,T_P)$ which is representable and can be described by \eqref{M-free-monoid}.  

Let $\Lambda'' = \Lambda_P/\Lambda$, and choose a free monoid $\cone''\subset\Lambda''$ that contains the image of $\cone_X$. (Notice that the image of $\cone_X$ is necessarily strictly convex.) Thus, the pair $(\cone'',\Lambda'')$ also corresponds to a toric variety $X''$ and its moduli space $M''$ which is representable, and it follows from the definitions that 
we have cartesian diagrams 

\begin{equation} 
\begin{tikzcd}
X \arrow{r}  \arrow{d}[swap]{}
& X_P \arrow{d} \\
\{1\} \arrow{r} & X'' 
\end{tikzcd}
\end{equation}

\noindent and

\begin{equation}\label{bc}
\begin{tikzcd}
M_X \arrow{r}  \arrow{d}[swap]{}
& M_P \arrow{d} \\
\{m''_0\} \arrow{r} & M'' 
\end{tikzcd}
\end{equation}
where $m''_0$ is the $k$-point of $M''$ corresponding to the trivial $T''$-bundle equipped with a trivialization. Since $M''$ is a separated algebraic space locally of finite type over $k$, the map $\{m''_0\} \to\M''$ is a closed embedding, and so is the map $M_X\to M_P$.
\end{proof}

The argument in the previous proof used the closed embedding $X\to X_P$ to deduce that $M_X\to M_P$ is also a closed embedding, but in fact we can have an embedding $M_X\hookrightarrow M_Y$ even if $X\to Y$ is not closed or is injective, as the following proposition shows.

\begin{proposition}\label{M-closed-embedding} 
Consider a morphism $(X,T)\to (Y,A)$ of toric varieties in $\cT$, represented by a morphism of their cocharacter monoids and groups: $(\cone_X, \Lambda_X)\to (\cone_Y,\Lambda_Y)$.

If the map $\cone_X\to \cone_Y$ is an isomorphism, the induced morphism $M_X\to M_Y$ is an equivalence.

If the map $\cone_X\to \cone_Y$ is injective, the induced morphism of schemes: $M_X\to M_Y$ is a closed embedding.
\end{proposition}

\begin{proof}
 For the first statement, we have $\Lambda_Y=\Lambda_X\times \Z^r$ as groups (for some $r$) and hence 
$$ A=T\times \Gm^{r}$$
and
$$ Y=X\times \Gm^{r}$$
(compatibly). 

For a test scheme $S$ and a pair $(E,\phi)$ consisting of an $A$-bundle $E$ over $C_S$ and a section $\phi: C_S \to Y\wedge^A E$, the projection of $\phi$ to $E \wedge^A \Gm^{r}$ defines a trivialization of the reduction of $E$ to a $\Gm^r$-bundle, therefore the map which takes any pair $(E',\phi')$ for $M_X$ to the pair $(E,\phi)=(E'\times \Gm^{r}, \phi' \times 1)$ for $M_Y$ is an equivalence of functors.

For the second statement, it is now enough to assume that $\Lambda_X$ is generated by $\cone_X$ and in particular that it also injects into $\Lambda_Y$. Then we use the last argument of the proof of the previous proposition to further reduce the statement to the case that $\Lambda_X\otimes_\Z \Q = \Lambda_Y\otimes_\Z \Q$. For, in general, we may replace $\Lambda_Y$ by $\Lambda_Z = \Lambda_Y \cap \Lambda_X\otimes \Q$ and $\cone_Y$ by $\cone_Z$ via a base change diagram as \eqref{bc}. 

Hence, from now on we assume that $\Lambda_X\otimes_\Z \Q = \Lambda_Y\otimes_\Z \Q$. We will show that the morphism $M_X\to M_Y$ is injective at the level of $S$-points for every test scheme $S$, and then that it is proper. This will prove that it is a closed embedding.

To show injectivity, assume that $(E,\phi)$ and $(E',\phi')$ in $M_X(S)$ with the same image in $M_Y(S)$. Notice that the injection $\Lambda_X\to \Lambda_Y$ corresponds to a map of tori: $T\to A$ with finite kernel $F$. Thus, there is an $F$-torsor $D$ over $C\times Y$ such that $E' = E\wedge^F D$, and an open $U\subset C\times S$ which surjects to $S$, such that the quotient between $\phi'$ and $\phi$ defines a trivialization of $D$ over $U$. Since $F$ is finite, this extends to a trivialization of $D$ over $C\times S$, which shows that $(E,\phi)\simeq (E',\phi')$.

Finally, we use the valuative criterion to prove properness: If $(R,\mathfrak m)$ is a discrete valuation ring with fraction field $K$, and $(E,\phi)\in M_Y(R)\cap M_X(K)$ (notice that by injectivity this last statement makes sense), then we claim that $(E,\phi)\in M_X(R)$. Indeed, let $U\subset C\times\Spec R$ denote the open subset over which $\phi$ has image in $A\wedge^A E\subset Y\wedge^A E$ and let $\tilde U = U\cup C\times \Spec K$ -- its complement has codimension greater or equal than $2$ in $C\times \Spec R$. 

The section $\phi$ defines a trivialization of $E$ over $U$. Moreover, by assumption, this trivialization extends to a reduction of $E$ to a $T$-bundle $\tilde E$ over $C\times\Spec K$. Thus, altogether, we have a $T$-bundle $\tilde E$ over $\tilde U$, and because $T$ is normal and the complement of $\tilde U$ has codimension at least $2$, this extends uniquely to a $T$-bundle $\tilde E$ over $C\times\Spec R$. Similarly, the section $\phi$ lifts by assumption to a section into $X\wedge^T \tilde E$ over $C\times \Spec K$, and hence to a section of $X\wedge^T \tilde E$ over $\tilde U$. By Hartogs' principle, since $X$ is normal, this extends to a section of $X\wedge^T\tilde E$ over the whole $C\times\Spec R$. This proves properness.
\end{proof}

Because of the proposition, one can think of the moduli space $M$ as being attached not to the toric variety $X$ itself, but to its associated monoid of cocharacters $\cone$. In fact, since this moduli space  specializes to the scheme of effective divisors (when $\cone=\N$), it is natural to think of $M$ as the scheme of ``$\cone$-valued divisors'', a notion that we now introduce.

A {\em $\cone$-valued divisor} on $C$ is
a formal sum $D=\sum_x \lambda_x x$  where $x$ runs over the set of  closed points of $C$ and $\lambda_x\in \cone$ with $\lambda_x=0$ for all but finitely many $x$. 

One can attach to each $k$-point of $M$ a $\cone$-valued divisor on $C$. Each $k$-point of $M$ corresponds to a map $\phi:C\to [X/T]$ such that the preimage $U=\phi^{-1}[T/T]$ is a nonempty subset of $C$. Over $U$, the underlying $T$-bundle $E$ of $\phi$ is trivialized by $\phi$. 
At each point $x\in C-U$, $\phi$ defines a coset in $X(\cO_c)/T(\cO_c)$ thus an element $\lambda_x\in \cone$ according to \eqref{orbit-torus} (with $\cO$ replaced by $\cO_c$). We define $\sum_{x\in C-U} \lambda_x x$ to be the $\cone$-valued divisor associated to $\phi$. 

\begin{lemma} \label{point-as-divisor}
The above construction induces a canonical bijection between $M(k)$  and the set of $\cone$-valued divisors on $C$.  
\end{lemma}

\begin{proof} 
The converse construction is based on a variant of Beauville-Laszlo's formal patching theorem proved in \cite{Heinloth}. One can construct a $k$-point of $M$ attached to each $\cone$-valued divisors on $C$. Let $D=\sum \lambda_x x$ be a $\cone$-valued divisor on $C$. The associated point $\phi:C\to [X/T]$ can be represented as $T$-torsor $E$ over $C$ equipped with a section $\phi:C\to X\wedge^T E$. The pair $(E,\phi)$ can be obtained via formal patching from the following collection of data:
\begin{itemize}
\item over the open subset $C'$ complement of the finite set of points $x\in |C|$ where $\lambda_x\not=0$, we set $E'=T$ to be the trivial $T$-torsor and $\phi':C'\to X$ the constant section $\phi'=1$;
\item over each formal disc $C_x$ with uniformizing parameter $t_x$, we set $E_x$ to be the trivial $T$-torsor and $\phi_x= t_x^{\lambda_x} \in X(\cO_x)$;
\item there is a unique way to glue $(E',\phi')$ with $(E_x,\phi_x)$ over the punctured formal disc $C_x^\punc$.
\end{itemize}
One can check that the above constructions give rise to reciprocal bijections between the set of $k$-points of $M$ and the set of $\cone$-valued divisors on $C$. 
\end{proof}

A {\em multiset} in $\cone$ is an element $\mu$ of the free monoid generated by $\cone-\{0\}$. A multiset in $\cone$ will be written as a sum 
\begin{equation} \label{multiset}
\mu=\sum_{\lambda\in \cone} \mu(\lambda) e^\lambda \in \bigoplus_{\lambda\in\cone -\{0\}} \N e^\lambda.
\end{equation}
with the convention that $\mu(0)=0$. One can attach to each $\cone$-valued   divisor $D=\sum\limits_{x\in Z} \lambda_x x$ a  multiset 
\begin{equation}
\mu_D(\lambda)=\sum_{\lambda_x=\lambda} \deg(x)e^{\lambda}.
\end{equation}

There is a natural order on the set of multisets:  we will say that $\mu$ refines $\mu'$ and write $\mu\vdash \mu'$ if the difference $\mu-\mu'$ viewed as element of $\bigoplus_{\lambda\in\cone-\{0\}} \Z e^\lambda$ can be written as a sum of elements of the form $e^{\lambda'}+e^{\lambda''} - e^\lambda$ with $\lambda'+\lambda''=\lambda$ in $\cone$.  

We define the degree of a multiset $\mu$ as in \eqref{multiset} to be \begin{equation} \label{degree}
\deg(\mu)=\sum_{\lambda\in\cone} \mu(\lambda) \lambda \in\cone.
\end{equation}
The degree map defines a homomorphism of monoids $\bigoplus_{\lambda\in\cone-\{0\}} \N e^\lambda \to \Lambda$ whose fibers are finite sets. 
We define the degree of a $\cone$-valued divisor $D=\sum_x \lambda_x x$ by 
$$\deg(D)=\sum_x \lambda_x \deg(x).$$ 
It is easy to see that the degree of a $\cone$-valued divisor coincides with the degree of its multiset.

\begin{proposition} \label{stratification-toric}
There exists a unique stratification of $M$
\begin{equation}
M=\bigsqcup_\mu M_\mu
\end{equation}
in strata indexed by multisets $\mu\in \bigoplus_{\lambda\in\cone}\N e^\lambda$ such that:
\begin{enumerate}
\item $M_\mu(k)$ is the set of $k$-points of $M$ whose associated $\cone$-valued divisors have multiset $\mu$.
\item $M_\mu$ is isomorphic with an open subset $U_\mu \subset \prod C_{\mu(\lambda)}$, where $\lambda$ runs over the finite subset of $\cone$ where $\mu$ takes positive values, $C_{\mu(\lambda)}$ is the 
$\mu(\lambda)$-th symmetric power of $C$ classifying effective divisors $D_\lambda$ of degree $\mu(\lambda)$, and the open subset $U_\mu$ is defined by the condition that $\sum_\lambda D_\lambda$ is a multiplicity free divisor.

\item $M_{\mu'}$ lies in the closure of $M_\mu$ if and only if $\mu 
\vdash \mu'$.
\end{enumerate}
\end{proposition}

\begin{proof}
First we describe one-dimensional strata $M_\mu$ indexed by a simple multiset $\mu=e^\lambda$ with $\lambda\in \cone$. We will prove that $M_\mu\simeq C$. 
\begin{itemize}
\item In the case $X=\bA^1$ and $\lambda=1$, this stratum classifies pairs $(E,\phi)$ where $E$ is a line bundle of degree $1$ over $C$ and $\phi$ is a nonzero global section. 
In this case there is an isomorphism $C\simeq M_\mu$ defined by $x\mapsto (\cO_C(x),1)$. 
\item In general, each element $\lambda$ of $\cone$ defines a map $(\bA^1,\Gm) \to (X,T)$ and induces a morphism $C\to M$. This morphism is a closed embedding by Proposition \ref{M-closed-embedding}. It defines an isomorphism from $C$ to a closed subscheme of $M$ which will be denoted by $M_{e^\lambda}$. For later use, we will denote this isomorphism
\begin{equation}
x\mapsto \phi_{\lambda x} \in M_{e^\lambda}.
\end{equation}
We have thus constructed minimal strata $M_{e^\lambda}$ which satisfy the first and second assertions of the Proposition. 
\end{itemize}

We now observe that, as $X$ is a normal affine scheme, the action of $T$ on $X$ extending the action of $T$ on itself can be extended to an algebraic monoid structure on $X$. It follows there exists a monoidal structure on $M$: for any points $\phi,\phi'\in M$ one can construct a third point $\phi\otimes \phi' \in M$  such that the underlying $T$-bundles satisfy $E=E'\wedge^T E''$. 

For $\lambda_\bullet=(\lambda_1,\ldots,\lambda_n)$ an element of $\cone^n$ for some $n\in\N$, one can define a map
\begin{equation} \label{psi-lambda-bullet}
\iota_{\lambda_\bullet}: C^n \to M
\end{equation}
by 
\begin{equation} \label{tensor-toric}
(x_1,\ldots,x_n) \mapsto \phi_{\lambda_1 x_1}\otimes 
\cdots \otimes \phi_{\lambda_n x_n}.
\end{equation} 
The map  \eqref{psi-lambda-bullet} is proper since $C^n$ is proper and $M$ is representable by a scheme.

Let $\mu=\sum_{i=1}^n e^{\lambda_i}$ denote the multiset associated with $\lambda_\bullet$. We denote $\overline M_{\mu}$ the image of 
$\iota_{\lambda_\bullet}$ which is a closed subscheme of $M$ as 
$\iota_{\lambda_\bullet}$ is a proper morphism. We also observe that $\mu$ determines $\lambda_\bullet$ up to reordering. Nevertheless the reordering has no effect on the image by \eqref{tensor-toric}, and therefore the image of $\iota_{\lambda_\bullet}$ depends only on $\mu$. We denote by $M_\mu$ the image of $C^{n\circ}$ -- the disjoint locus of $C^n$ -- in $\overline M_\mu$. It is easy to check that $M_\mu$ is an open subset of $\overline M_\mu$ and $C^{n\circ}$ is its preimage. 

With these geometric ingredients now set up, we can go on and prove the three assertions of Proposition \ref{stratification-toric}:

\begin{enumerate}
\item[1.] The first assertion is clear from the formula \eqref{tensor-toric}.

\item[3.] Since $C^n$ is proper and irreducible, its image by $\iota_{\lambda_\bullet}$ is an irreducible closed subset of $M$. Thus $\overline M_\mu$ is the closure of $M_\mu$. It now follows from \eqref{tensor-toric} that $M_{\mu'} \subset \overline M_\mu$ if and only if $\mu\vdash \mu'$.

\item[2.] After reordering, we can suppose that 
\begin{equation}
\lambda_\bullet=(\lambda_1,\ldots,\lambda_1,\lambda_2,\ldots,\lambda_2,\ldots,\lambda_r)
\end{equation}
where $\lambda_1,\lambda_2\dots,\lambda_r$ are distinct elements of $\cone$ with $\lambda_i$ appearing $\mu(\lambda_i)$ times. The morphism 
$\bar \iota_{\lambda_\bullet}: C^{n \circ} \to M_\mu$ is a finite surjective morphism invariant under the action of $\fS_{\mu(\lambda_1)} \times \cdots \fS_{\mu(\lambda_r)}$ where $\fS_d$ denotes the symmetric group of rank $d$. It follows that $\bar \iota_{\lambda_\bullet}: C^{n \circ} \to M_\mu$ factors through a finite morphism $U_\mu \to M_\mu$, $U_\mu$ being defined in the statement of Proposition \ref{stratification-toric}. We also know that the morphism $U_\mu \to M_\mu$ induces bijection on geometric points. We need to prove that it is an isomorphism.

We would be done if we knew either that $M_\mu$ is normal or that $\bar \iota_{\lambda_\bullet}: C^{n \circ} \to M_\mu$ is flat. However we just know that $M_\mu$ is integral for it is defined as the image of a finite morphism from an integral scheme. Nevertheless, in the case of the free monoid $\cone_P$, we know that the morphism $U_\mu \to M_{P,\mu}$ is an isomorphism by direct calculation. The map $\cone \to \cone_P$, induces by functoriality a map $M_\mu \to M_{P,\mu}$. 
We then derive a section of the morphism  $U_\mu\to M_\mu$. It follows that $U_\mu\simeq M_{\mu}$ as $U_\mu$ is normal.
\end{enumerate}

Finally we prove that $M(\bar k)=\bigsqcup_\mu M_\mu(\bar k)$ for a separable closure $\bar k$ of $k$. According to Lemma \ref{point-as-divisor}, $\bar k$-points on $M$ correspond bijectively with $\cone$-valued divisors of $C\otimes_k \bar k$. Its remains to check that $\bar k$-points on $M_\mu$ correspond bijectively with $\cone$-valued divisors of $C\otimes_k \bar k$ of type $\mu$. This follows easily from the very construction of $M_\mu$. 
\end{proof}

\begin{corollary} \label{components-toric}
\begin{enumerate}
\item Each connected component of $M$ contains a unique closed stratum of the form $M_\mu$ where $\mu=e^\lambda$ for some $\lambda\in\cone$. In particular, there is a canonical bijection $\pi_0(M)=\cone$, and each stratum $M_\mu$ lies in the connected component $M^\lambda$ indexed by $\lambda=\deg(\mu)$.

\item Irreducible components of $M$ are the closures of strata $M_\mu$ where $\mu$ are primitive multisets i.e.\ multisets of the form $\mu=\sum \mu(\nu)e^\nu$ with $\mu(\nu)=0$ unless $\nu$ is a primitive element of the cone $\cone$.
\end{enumerate}
\end{corollary}

\begin{proof}

\begin{enumerate}
\item[1.] Minimal elements for the partial order $\vdash$ are multisets of the form $\mu=e^\lambda$ for some $\lambda\in\cone$. Thus each stratum $M_\mu$ contains a unique closed stratum in its closure that is $M_{e^{\deg(\mu)}}$. 
Thus two strata $M_\mu$ and $M_{\mu'}$ belong to the same connected components of $M$ if and only if $\deg(\mu)=\deg(\mu')$. 
\item[2.]
Maximal multisets are those that cannot be further refined i.e.\ multisets of the form $\mu=\sum \mu(\nu)e^\nu$ with $\mu(\nu)=0$ unless $\nu$ is a primitive element of $\cone$. Indeed if $\mu=\sum \mu(\nu)e^\nu$ with $\mu(\nu)>0$ for some nonprimitive element $\nu$, then $\mu$ can be refined by replacing the term $e^\nu$ with $e^{\nu'}+e^{\nu''}$ where $\nu',\nu''$ are elements of $\cone$ adding up to $\nu$. 
\end{enumerate}
\end{proof}

\begin{corollary} \label{resolution-singularity}
For each primitive multiset $\mu=\sum \mu(\nu)e^\nu$, we set
$$\mu_\bullet=(\nu_1,\ldots,\nu_1,\nu_2,\ldots,\nu_2,\ldots,\nu_r)$$
where $\nu_1,\nu_2,\ldots,\nu_r$ are distinct primitive elements of $\cone$, $\nu_i$ appearing $\mu(\nu_i)$ times. The map $\iota_{\mu_\bullet}$ defined in \eqref{tensor-toric} gives rise to the normalization of $\overline M_{\mu}$ of the form
\begin{equation}\label{normalization}
\iota_\mu:C_\mu=C_{\mu(\nu_1)} \times \cdots \times C_{\mu(\nu_r)} \to \overline M_\mu.
\end{equation}

Moreover, the disjoint sum of $\iota_\mu$ ranging over all primitive multisets $\mu$ defines a normalization of $M$. This normalization turns out to be a resolution of singularities.
\end{corollary}

\begin{proof}
The morphism 
$$\bar\iota_{\mu_\bullet}: C^n \to \overline M_\mu$$
defined in \eqref{tensor-toric} is invariant under $\fS_{\mu(\nu_1)} \times\cdots\times \fS_{\mu(\nu_r)}$ and therefore factors through a finite morphism 
\begin{equation}
\iota_\mu:C_{\mu(\nu_1)} \times \cdots \times C_{\mu(\nu_r)} \to \overline M_\mu.
\end{equation}
During the proof of assertion 2 of Proposition \ref{stratification-toric}, we have seen that this morphism is an isomorphism over the dense open subset $M_\mu$. It is thus the normalization of $\overline M_\mu$.

Since $C_\mu$ is smooth, the disjoint sum of $\iota_\mu:C_\mu\to \overline M_\mu$ is a resolution of singularities of $M$. 
\end{proof}

\begin{proposition}
Let $k$ be a finite field. The IC-function $M(k) \to \Ql$ can be expressed as a formal series 
\begin{equation} \label{IC-M-toric}
\IC_M= \sum_{D=\sum_x \lambda_x x} a_D \prod_x e_x^{\lambda_x}
\end{equation}
where $D$ runs over the monoid of $\cone$-valued divisors on $C$, $a_D$ is the trace of Frobenius on the stalk of $\IC$ over the $k$-point of $M$ corresponding to $D$. Moreover, there is an equality of formal series:
\begin{equation} \label{global-equality-toric}
\IC_M = \prod_{x\in |C|} \prod_{\nu\in\Prim(\cone)} (1-e_x^\nu)^{-1}.
\end{equation}
\end{proposition}

\begin{proof}
We have constructed the normalization of $M$ which is at the same time a normalization of singularity
\begin{equation}
\bigsqcup \iota_\mu : \bigsqcup C_\mu \to M
\end{equation}
where $\mu$ ranges over the set of primitive multisets. It follows that the coefficient $a_D$ in the series \eqref{IC-M-toric} is just the number of $k$-points in the fiber of $\bigsqcup \iota_\mu$ over the point $\phi_D\in M(k)$ corresponding to the $\cone$-valued divisor $D$. 

It is easy to check that if $D=\sum_x \lambda_x x$, we have 
\begin{equation}
a_D=\prod_{x\in X} m(\lambda_x)
\end{equation}
where $m(\lambda_x)$ is the number of ways of writing $\lambda_x$ as a sum of primitive elements of $\cone$. We thus derive \eqref{global-equality-toric}.
\end{proof}

Theorem \eqref{local-equality-toric} follows from the above proposition, according to formula \eqref{global-local} and Proposition \ref{existence-of-global-point}.

\begin{remark}
It is neither true, in general, that the closure of $M_\mu$ is isomorphic to the product $\prod C_{\mu(\lambda)}$ of Proposition \ref{stratification-toric}, nor that the irreducible components of $M$ intersect transversely. 

For example, let us (for simplicity) assume that the curve $C=\bP^1$, and let us consider the connected component $M^{\lambda}$ of $M$ of $\c$-valued divisors of degree equal to a fixed $\lambda\in \c$. 

Consider the dense open subvariety $M'\subset M^{\lambda}$ of those $\c$-valued divisors whose support is contained in $\bA^1\subset \bP^1$. Symmetric powers of $\bA^1$ are also affine spaces with coordinates given, when the characteristic is large compared to the power, by elementary symmetric polynomials. Thus, assuming that the characteristic of the field is large enough, we can embed $M'$ in an affine space with coordinates: $$Z_\chi^j (D) = \sum_{x\in \bA^1} \left<\lambda_x, \chi\right> \cdot x^j.$$ 
Here, $D$ is the $\c$-valued divisor $D = \sum \lambda_x x $ (with $x\in \bA^1$), $\chi$ varies in a set of generators of the dual monoid $\c^*$, as before, and $j$ varies over a large enough set of positive integers. 

The numbers $Z_\chi^j$ are obviously determined by the ``coordinates''
$$ Z^j(D) = \sum_{x\in \bA^1} x^j\cdot \lambda_x $$
which are valued in an affine space whose $R$-valued points are $\Lambda\otimes R$. Thus, we have embedded $M'$ into a product of such affine spaces.

It is now easy to construct examples where the maps \eqref{normalization} are not embeddings. For instance, for $\c$ the monoid generated by $(3,0)$, $(2,1)$, $(1,2)$ and $(0,3)$ in $\Z^2$ and $\lambda = $ the sum of those generators, the restriction of the map \eqref{normalization} to $(\bA^1)^4$:
$$ (\bA^1)^4 \to M'$$
does not have an 1-1 differential at $(0,0,0,0)$ (or, for that matter, any point on the diagonal). Indeed, the differential of $Z^j$ is zero for $j\ge 2$, and the differential of $Z^1$ is non-injective.

Similarly, in the same example, considering the partitions $(3,3)= (3,0)+(0,3) = (2,1)+(1,2)$, it is easy to see that the corresponding irreducible components (which are now isomorphic to $C^2$) share the same tangent space over the diagonal, and therefore do not intersect transversely.

\end{remark}

\section{$L$-monoid}

In \cite{Ngo-monoids}, a special class of affine embeddings has been emphasized for their connection with the test functions defining unramified local $L$-factors. These embeddings are constructed in the following situation: $G$ is a reductive group equipped with a "determinant" whose kernel $G'$ is simply connected
\begin{equation}
0 \to G' \to G \xrightarrow{\det} \Gm \to 0
\end{equation}
The center of its complex dual group $\hat G$ is then $\Gm$:
\begin{equation}
0 \to \Gm \to \hat G \to \hat G' \to 0.
\end{equation}
Let $\rho:\hat G \to \GL(V_\rho)$ be an irreducible representation of its complex dual group. We assume that the diagram
\begin{equation}
\begin{tikzcd}
\Gm \arrow{r}  \arrow{d}[swap]{\id}
& \hat G \arrow{d} \\
\Gm \arrow{r} & \GL(V_\rho) 
\end{tikzcd}
\end{equation}
is commutative. In other words, the central $\Gm$ of $\hat G$ acts on the vector space $V_\rho$ as scalar. 

In this setting, one can construct a normal affine embedding $X$ of $G$ fitting into a commutative diagram:
\begin{equation} \label{X}
\begin{tikzcd}
G \arrow{r}  \arrow{d}[swap] {\det}
& X \arrow{d} \\
\Gm \arrow{r} & \bA^1 
\end{tikzcd}
\end{equation}
Since $X$ is normal and affine, the left and right actions of $G$ on $X$ merge into a monoidal structure of $X$. The construction of the monoid $X$ is based on Vinberg's theory of the universal monoid \cite{Vinberg} that we now recall. 

The universal flat monoid $X^+$ is an affine embedding of $G^+$, where $G^+$ is an entension of a torus $T^+$ by $G'$, $r$ being the rank of $G'$
$$0\to G' \to G^+ \to T^+ \to 0.$$
Let $T'$ be a maximal torus of $G'$. Following Vinberg, we set $G^+=(G'\times T')/Z'$ where $Z'$ is the center of $G'$ acting antidiagonally on $G'$ and $T'$. It follows that $T^+=T'/Z'$ is the maximal torus of the adjoint group that can be identified with $\Gm^r$ with aid of the set of simple roots $\{\alpha_1,\ldots,\alpha_r\}$ associated with the choice of a Borel subgroup of $G'$ containing $T'$.

Let $\omega_1,\ldots,\omega_r$ denote the fundamental weights dual to the simple coroots $\alpha^\vee_1,\ldots,\alpha^\vee_r$ and let $\rho'_i:G'\to \GL(V_i)$ denote the irreducible representation of highest weight $\omega_i$. This can be extended to $G^+$
\begin{equation} \label{rho-i+}
\rho_i^+ : G^+ \to \GL(V_{i})
\end{equation}
by the formula $\rho_i^+(t,g)=\omega_i(t) \rho'_{i}(g)$ where $w_0$ is the long element in the Weyl group $W$ of $G$. The root $\alpha_i:T\to \Gm$ will also be extended to $G^+$
$$\alpha_i^+ :G^+ \to \Gm$$
by $\alpha_i^+(t,g)=\alpha_i(t)$. All together, these maps define a homomorphism
\begin{equation} \label{alpha+rho+}
(\alpha^+,\rho^+) : G^+ \to \Gm^r \times \prod_{i=1}^r \GL(V_i).
\end{equation}
If the characteristic is large enough, Vinberg's universal monoid $X^+$ is defined as the closure of $G^+$ in $\bbA^r \times \prod_{i=1}^r \End(V_i)$. In small characteristic, it is defined to be the normalization of this closure, see \cite{Vinberg} and \cite{Rittatore}.

The universal monoid $X^+$ fits into a commutative diagram
\begin{equation}\label{X+}
\begin{tikzcd}
G^+ \arrow{r}{\det}  \arrow{d} 
& X^+ \arrow{d} \\
\Gm^r \arrow{r} & \bA^r
\end{tikzcd}
\end{equation}
The diagram \eqref{X} is to be obtained from \eqref{X+} by base change. The highest weight of the irreducible representation $\rho:\hat G\to \GL(V_\rho)$ defines a cocharacter $\lambda:\Gm \to T$ where $T$ is the maximal torus of $G$ and induces a homomorphism $\lambda_\ad:\Gm \to \Gm^r$. For $\lambda$ is dominant, $\lambda_\ad$ can be extended to a morphism of monoids
\begin{equation}\begin{tikzcd}
\Gm \arrow{r} \arrow{d}[swap]{\lambda_\ad} 
& \bA^1 \arrow{d}{\lambda_\ad} \\
\Gm^r \arrow{r} & \bA^r
\end{tikzcd}
\end{equation} 
By base change with respect to $\lambda_\ad:\bA \to \bA^r$, we obtain the affine embedding $X$ of $G$. 

Let $\cL X$ denote the formal arc space of $X$, $\cL^\circ X$ denote the open subset of non-degenerate arcs with respect to the open subset $X^\circ=G$ defined as in \eqref{circ}. We have 
\begin{equation}
\cL^\circ X(k) = X(\cO) \cap G(F) 
\end{equation}
where $\cO=k[[t]]$ and $F=k((t))$. If $k$ is a finite field, the IC-function of $\cL^\circ X$ is a left and right $G(\cO)$-invariant function on $X(\cO) \cap G(F)$. There is a unique way to decompose the function $\IC_{\cL^\circ X}$ by support
\begin{equation}
\IC_{\cL^\circ X}=\sum_{n=0}^\infty \psi_n
\end{equation}
where $\psi_n$ is a function supported on the compact set
\begin{equation}
\{g\in X(\cO)\cap G(F) \mid \val(\det(g))=n\}.
\end{equation}
Each $\psi_n$ is a compactly supported, left and right $G(\cO)$-function on $G(F)$, thus is an element of the spherical Hecke algebra of $G(F)$. The value that $\psi_n$ takes on different double cosets may be rather complicated, see  \cite{WWL} but its Satake transform can be described simply.

\begin{theorem}\label{local-identity-monoid}
We have 
$\IC_{\cL^\circ X}=\sum_{n=0}^\infty \psi_n$
where $\psi_n$ is the function in the spherical Hecke algebra of $G(F)$ whose Satake transform is the function 
$$\hat\psi_n(\sigma)=\tr(\sigma,\Sym^n \rho)$$ 
for all $\sigma\in \hat G$.
\end{theorem}

As in the toric case, in order to prove this local identity, we need to consider its global analogue.
Let $C$ be a smooth projective curve over a field $k$. We consider the algebraic stack $\Map(C,[X/G])$ and the open substack $M$ of maps $\phi:C\to [X/G]$ such that restricted to a nonempty open subset $U$ of $C$, $\phi|_U:U \to [X/G]$ factors through $[G/G]$.
  
The determinant map \eqref{X} gives rise to a morphism 
\begin{equation}
f: M \to C_\N=\bigsqcup_{n=0}^\infty C_n
\end{equation}
where $C_\N$ classifies pairs $(L,\alpha)$ where $L$ is a line bundle and $\alpha$ is a global section of $L$ which generically induces a trivialization of $L$. Let $M_n$ denote the preimage of $C_n$ and
\begin{equation}
f_n:M_n \to C_n
\end{equation}
the restriction of $f$ to $M_n$. 

\begin{proposition}
Let $D\in C_n$ be an effective divisor of degree $n$ of $C\otimes_k \bar k$
\begin{equation}
D=\sum_{i=1}^m n_i c_i
\end{equation}
where $c_1,\ldots,c_m$ are distinct points of $C$, and $n_1,\ldots,n_m$ are natural numbers. Then there is a canonical closed embedding 
\begin{equation} \label{alpha-D}
z_D: f_n^{-1}(D) \to \prod_{i=1}^m \Gr_{c_i}
\end{equation}
where $\Gr_{c_i}$ is the affine Grassmannian of $G$ relative to the formal disc around $c_i$. Moreover, the image of $z_D$ is isomorphic to the product of closed Schubert varieties
\begin{equation} \label{Schubert}
z_D(f_n^{-1}(D))\simeq \prod_{i=1}^m \overline\Gr_{c_i, n_i \lambda}
\end{equation}
according to the indexation of Schubert varieties in $\Gr$ by dominant coweight as in Definition 2.10 of \cite{Var}.
\end{proposition}

\begin{proof}
Let $S$ be an arbitrary test $k$-scheme.
Let $(E,\phi)$ be a $S$-point of $M$ whose image by the determinant map is $(L_D,z_D)$ where $L_D=\cO(D)$ and $z_D$ is the canonical global section of $\cO(D)$ whose zero divisor is $D$. Since the restriction of $(L_D,z_D)$ to $C'_S$ is just the trivialized line bundle, the restriction of $\phi$ to $C'_S$ determine a trivialization $\phi'$ of $E$ over $C'_S$. 

According to Beauville-Laszlo's uniformization theorem, see \cite{Heinloth}, giving a principal bundle $E$ on $C_S=C\times S$, with a trivialization $C'_S=C'\times S$ is equivalent to giving for each $i=1,\ldots,n$ a principal $G$-bundle $\cE_i$ on $C_{c_i,S}=C_{c_i}\hat\times S$ equipped with a trivialization on $C_{c_i,S}^*$; in other words, an $S$-point on the affine Grassmannian $\Gr_{c_i}$. This defines the map $z_D$ of \eqref{alpha-D}. 

The requirement that the trivialization $\phi'$ of $E$ over $C'_S$ comes from a $X$-map $\phi:E_0\to E$ in the sense of Section 2, can be expressed in terms of associated vector bundles. Recall that we have the representations $\rho_i:G\to \GL(V_i)$ deduced from the representations $\rho_i^+: G^+ \to \GL(V_i)$ defined in \eqref{rho-i+}. By the construction of the universal monoid $M^+$, the trivialization $\phi'$ of $E$ extends to $X$-map $\phi:E_0\to E$ if and only if for all $i=1,\ldots,r$, the trivialization of the vector bundle $\rho_i(E)$ over $C'_S$, deduced from $E$ via the representation $\rho_i$, can be extended to a 
$\cO_{C_S}$-linear map
\begin{equation}
\cO_{C_S} \to \rho_i(E)\left(\sum_{i=1}^m
\langle \omega_i, n_i\lambda \rangle c_i \right).
\end{equation}
Now, this is equivalent to saying that $\cE_i$ is a $S$-point in the Schubert cell $\overline\Gr_{c_i,n_i\lambda}$ in the definition of \cite{Var}.
\end{proof}

A remark of caution is in order about different definitions of Schubert cells in the affine Grassmannian. In \cite{MV}, Mirkovic and Vilonen define the Schubert cell $\overline\Gr_{c_i,n_i\lambda}^{\rm MV}$ to be the closure of the orbit $\Gr_{c_i,n_i\lambda}$ of $\cL G$ on $\Gr$. The Schubert cell $\overline\Gr_{c_i,n_i\lambda}$, constructed as a functor as in \cite{Var}, defines a closed subscheme of $\Gr$ which may or may not coincide with $\overline\Gr_{c_i,n_i\lambda}^{\rm MV}$, but they have the same underlying topological space. Since we are interested in $\ell$-adic sheaves, this difference doesn't matter, though of course it would be nice to prove that the two definitions give rise to the same closed subscheme of $\Gr$. We won't prove this in the present paper.

Another remark of caution is the following. We infer from the proposition that there is a canonical isomorphism between the set $M(k)$ of $k$-points of $M$ and the restricted product
\begin{equation} \label{restricted-product}
\prod_{x\in X}' (M(\cO_x) \cap G(F_x))/G(\cO_x)
\end{equation}
consisting in a collection of cosets $(g G(\cO_x), x\in |X|)$ with 
$g\in M(\cO_x) \cap G(F_x)$ for all $x$, and $g\in G(\cO_x)$ for almost all $x$. In particular, there is a canonical map 
\begin{equation} \label{pr-x}
\pr_x: M(k) \to \Gr_x(k)
\end{equation}
from $M(k)$ to the set of $k$-points of the affine Grassmannian $\Gr_x$ at all $x\in |C|$. However, this map doesn't derive from a well-defined morphism $M\to \Gr_x$.

\begin{theorem} \label{symmetric-convolution}
 
The restriction of the intersection complex $\IC_{M_n}$ of $M_n$ to every geometric fiber of $f_{n}$ is still a perverse sheaf. This perverse sheaf can be described as follows. If $D =\sum_{i=1}^m n_i c_i$ is an effective divisor of degree $n$ of $C$, where $c_1,\ldots,c_m$ are distinct points of $C$, and $n_1,\ldots,n_m$ are natural numbers, then 
\begin{equation}
z_{D,*}(\IC_{M_n} |_{f_n^{-1}(D)}) = \boxtimes_{i=1}^m K_i
\end{equation}
where $K_i$ is an equivariant perverse sheaf over $\Gr_{c_i}$ whose Satake transform is $\Sym^{n_i}(\rho)$.
\end{theorem}

\begin{corollary}\label{global-equality-monoid}
 
If $k$ is a finite field, the IC-function on $M(k)$, can be expressed as 
\begin{equation} \label{prod-psi}
\IC_M=\prod_{x\in |X|} \sum_{n=0}^\infty \psi_{x,n}
\end{equation}
where $\psi_{x,n}$ is an element of the spherical Hecke algebra of $G(F_x)$ characterized by the property 
$\hat\psi_{x,n}(\sigma)=\tr(\sigma,\Sym^n \rho)$. 
The infinite product in equality \eqref{prod-psi} makes sense as a function on the restricted product \eqref{restricted-product}.
\end{corollary}

We consider the moduli space $M^n$ of chains 
\begin{equation}\label{chain}
\begin{tikzcd}
E_0 \arrow{r}{\phi_1} 
& E_1 \arrow{r}{\phi_2} & \cdots \arrow{r}{\phi_n} & E_n\end{tikzcd}
\end{equation}
where $E_0$ is the trivial $G$-bundle, and $\phi_i: E_{i-1} \to E_i$ is a $X$-morphism from $E_{i-1}$ to $E_i$, as defined in Section 2, whose determinant is 
$$\det(\phi_i)=\cO_C(x_i)$$
with $x_i\in C$. Thus we have a morphism $M^n\to C^n$ fitting in a commutative diagram 
 
\begin{equation} \label{convolution-diagram}
\begin{tikzcd}
M^n \arrow{r}{} \arrow{d}[swap]{\pi_M}
& C^n \arrow{d}{\pi_C} \\
M_n \arrow{r}[swap]{f_n}
& C_n
\end{tikzcd}
\end{equation}
where the left vertical map consist in forgetting all members of the chain \eqref{chain} but the last component $E=E_n$, and replacing $(\phi_1\ldots,\phi_n)$ by their composition.

\begin{proposition}
\begin{enumerate}
\item The diagram \eqref{convolution-diagram} is cartesian over the open subset $C_n^\circ$ of $C_n$ classifying multiplicity free divisors of degree $n$ on $C$. 

\item The morphism $\pi_M: M^n \to M_n$ is small in the stratified sense of Mirkovic and Vilonen \cite{MV}.

\item Over each divisor $D\in C_n$, the morphism $\pi_M: (f_n\circ \pi_M)^{-1}(C) \to f_n^{-1}(D)$ is semismall in the stratified sense of Mirkovic and Vilonen \cite{MV}.

\end{enumerate}
\end{proposition}

\begin{proof}
The first assertion amounts to the same to say that over the open subset 
$C^{n\circ}$, we can reconstruct the the whole chain \eqref{chain} from $E_0\to E_n$. This is a consequence of Beauville-Laszlo's formal 
patching theorem, as proved in \cite{Heinloth}.
The last two assertions are proved by the same argument as in \cite{MV}. 
\end{proof}

\begin{corollary}
\begin{enumerate}
\item The direct image $(\pi_M)_*{\rm IC}(M^n)$ is a perverse sheaf which is isomorphic to the intermediate extension of its restriction to the open subset $M_n^\bullet=M_n \times_{C_n} C_n^\circ$. In particular, it is equipped with an action of the symmetric group $\fS_n$. 
\item There exists an isomorphism of perverse sheaves
$$ ((\pi_M)_*{\rm IC}_{M^n})^{\fS_n} = {\rm IC}_{M_n}.$$
In particular, the restriction of $\IC_{M_n}$ to every fiber $f_n^{-1}(D)$ is a perverse sheaf.
\end{enumerate}
\end{corollary}

\begin{proof}[Proof of \ref{symmetric-convolution}]
We have shown that the restriction of $\IC_{M_n}$ to the fiber $f_n^{-1}(D)$ is a perverse sheaf, as the $\fS_n$-invariant part of the push-forward of the restriction of $\IC_{M^n}$ to $(f_n \circ \pi_M)^{-1}(D)$. In order to have a more precise description, we ought analyze the action of $\fS_n$ on this push-forward. Fortunately, this is well known thanks to the geometric Satake theory \cite{MV}. 

From direct investigation, one sees that the push-forward of the restriction of $\IC_{M^n}$ to $(f_n \circ \pi_M)^{-1}(D)$ can be identified with 
\begin{equation} \label{n-fold-convolution}
\Ind_{\fS_{n_1} \times\cdots\times \fS_{n_r}}^{\fS_n} (\boxtimes_{i=1}^r \cA_{\rho}^{* n_i})
\end{equation}
via the embedding $z_D$ of the fiber $f_n^{-1}(D)$ into a product of affine Grassmannians. Here $\cA_{\rho}$ denote the IC-complex of the Schubert variety indexed by $\rho$, and $\cA_{\rho}^{* n_i}$ is its $n_i$-fold convolution power. Moreover the action of $\fS_{n_1} \times\cdots\times \fS_{n_r}$ is given by the commutativity constraint, by essentially the very definition of the commutativity constraint provided by Mirkovic and Vilonen. 

The $\fS_n$-invariant factor of \eqref{n-fold-convolution} can be identified with the $\fS_{n_1} \times\cdots\times \fS_{n_r}$-invariant factor in $\boxtimes_{i=1}^r \cA_{\rho}^{* n_i}$. This direct factor has the form $\boxtimes_{i=1}^r K_i$ where $K_i$ corresponds to the representation $\Sym_{n_i}(V_\rho)$ of $\hat G$ via the geometric Satake equivalence. 
\end{proof}

We now derive Theorem \ref{local-identity-monoid} from Corollary \ref{global-equality-monoid} by using formula \eqref{global-local} and Proposition \ref{existence-of-global-point}.

\Addresses

\newpage

\title{Erratum to: ``On the formal arc space
of a reductive monoid''}

\author{A.\ Bouthier, B.C.\ Ng\^o, Y.\ Sakellaridis}

\maketitle

\setcounter{section}{4}

The main theorem about the IC function of the arc space of an $L$-monoid, Theorem 4.1 in our paper ``On the formal arc space of a reductive monoid'', has to be corrected as follows:
\begin{theorem}
We have 
$$\IC_{\cL^\circ X}=\sum_{n=0}^\infty q^{-n\left<\nu_G, \lambda \right>} \psi_n$$
where $\nu_G=\frac{1}{2}\sum_{\alpha>0}\alpha$ is half the sum of positive roots, and $\psi_n$ is the function in the spherical Hecke algebra of $G(F)$ whose Satake transform is the function 
$$\hat\psi_n(\sigma)=\tr(\sigma,\Sym^n \rho)$$ 
for all $\sigma\in \hat G$.
\end{theorem}
Other results in this section have to be modified accordingly, as we will outline below.

Notice that we have used the symbol $\rho$ for the representation of the dual group with highest weight $\lambda$, defining the $L$-monoid $X$, which is why we are denoting the half-sum of positive roots by $\nu_G$.

In terms of $L$-functions, this means that the Godement-Jacquet local zeta integral of $\IC_{\cL^\circ X}$ against the unramified matrix coefficient of an irreducible unramified representation $\pi$:
$$ \int_{G(k((t)))} \IC_{\cL^\circ X}(g) \left< \pi(g) v, \tilde v\right> dg$$
(where $v$, $\tilde v$ are dual unramified vectors with $\left< v, \tilde v\right> =1$) is equal, when convergent, to the local $L$-function:
$$ L(\pi, \rho, -\left<\nu_G, \lambda \right>).$$

This shift by $-\left<\nu_G, \lambda \right>$ is of course well-known -- in the original case of Godement and Jacquet, where $X=\Mat_n$, we have $\left<\nu_G, \lambda \right> = \frac{n-1}{2}$ -- but it escaped our attention because we failed to account for the correct normalization of the IC sheaf in the geometric Satake isomorphism.

In the geometric Satake isomorphism, $G(\cO)$-equivariant perverse sheaves on the affine Grassmannian are associated to representations of the dual group of $G$. For this to be compatible with the classical Satake isomorphism via the sheaf-function dictionary, the sheaves have to be pure of weight zero. For the IC complex of the Schubert stratum associated to the representation $\rho$ with highest weight $\lambda$, denoted $\cA_\rho$ in our paper, this means that it is the intermediate extension of the local system
$$ \overline{\QQ_\ell} [2\left<\nu_G, \lambda \right>] \left(\left<\nu_G, \lambda \right>\right)$$
over the smooth locus, which is known to be of dimension $2\left<\nu_G, \lambda \right>$. The notation $[\cdot]$ here stands for cohomological shift, and $(\cdot)$ stands for a Tate twist. When $\left<\nu_G, \lambda \right>$ is a half-integer but no integer, the Tate twist should be understood as tensoring by $\overline{\QQ_\ell} \left(\frac{1}{2}\right)^{2\left<\nu_G, \lambda \right>}$, where $\overline{\QQ_\ell} \left(\frac{1}{2}\right)$ is a chosen square root of $\overline{\QQ_\ell} (1)$.
Let us denote the square root of $q$ by which the geometric Frobenius automorphism acts on $\overline{\QQ_\ell} \left(-\frac{1}{2}\right)$ by $\alpha\in \overline{\QQ_\ell}$. The corresponding half-integral power of $q$ in the corrected statement of Theorem 4.1 should be interpreted as a complex number, \emph{in terms of an embedding $\overline{\QQ_\ell}\hookrightarrow \CC$ which sends $\alpha$ to $(-{\sqrt{q}})\in \CC$}.  This is to ensure compatibility between the geometric and the classical Satake isomorphisms, as we will see below. 

The perverse sheaf $\cA_\rho$ is the one that corresponds to the representation $\rho$ of the dual group, and via the sheaf-function dictionary it gives a function which is equal to $(-\alpha)^{-2\left<\nu_G, \lambda \right>}$ times the characteristic function of the double coset $G(\cO)\lambda(t)G(\cO)$, plus terms supported on Cartan double cosets of lesser weights. Via the map $\alpha\mapsto -\sqrt{q}$, this function becomes $q^{-\left<\nu_G, \lambda \right>}$ times the characteristic function of the double coset $G(\cO)\lambda(t)G(\cO)$ plus terms on other cosets, which is compatible with the classical Satake isomorphism, as commonly defined, s.\ \cite{Gross-Satake}.

Compare this to the normalization of $IC_{M^n}$ and $IC_{M_n}$ in our paper, which are taken to be constant in degree zero on the smooth locus. To account for the different normalizations, formula (4.24) should read:
\setcounter{equation}{23}
\begin{equation}
\Ind_{\fS_{n_1} \times\cdots\times \fS_{n_r}}^{\fS_n} \left(\boxtimes_{i=1}^r \left(\cA_{\rho} [-2\left<\nu_G, \lambda \right>] \left(-\left<\nu_G, \lambda \right>\right)   \right)^{* n_i}\right),
\end{equation}
and equation (4.3) should read:
\setcounter{equation}{2}
\begin{equation}
z_{D,*}(\IC_{M_n} |_{f_n^{-1}(D)}) = \boxtimes_{i=1}^m K_i [-2n_i\left<\nu_G, \lambda \right>] \left(-n_i\left<\nu_G, \lambda \right>\right),
\end{equation}
accounting for the required correction in Theorem 4.1.

Moreover, the term ``perverse sheaf'' in Theorem 4.3 and Corollary 4.6 should be taken to mean ``cohomological shift of a perverse sheaf''.

\Addresses

\end{document}